\documentclass{amsart}
\title[Fixed Point]{An addedum to the paper "Some elementary estimates for the Navier-Stokes system"}
\author{Jean Cortissoz}
\address{Departamento de Matem\'aticas,
Universidad de Los Andes\\
Bogot\'a DC, COLOMBIA}
\email{jcortiss@uniandes.edu.co}
\newtheorem{theorem}{Theorem}
\newtheorem{lemma}{Lemma}

\begin{document}
\begin{abstract}
In this paper we give a proof of the existence of global
regular solutions to the Fourier transformed Navier-Stokes system
with small initial data in $\Phi\left(2\right)$ via an iteration argument.
The proof of the regularity theorem is a minor modification of the proof given in the paper 
"Some elementary estimates for the Navier-Stokes system", so this paper is intended to
be just a complement to the afore mentioned paper.
\end{abstract}
\maketitle

\section{Introduction}

A Generalized Navier-Stokes system (with periodic boundary conditions
on $\left[0,1\right]^3$) is a system of the form
\begin{eqnarray}
\label{FourierNS}
v^k\left(\xi,t\right)&=&
\psi^k\left(\xi\right)\exp\left(-\left|\xi\right|^2 t\right)\\ \notag
&&
+\int_{0}^t\exp\left(-\left|\xi\right|^2\left(t-s\right)\right)\sum_{\mathbf{q}\in\mathbb{Z}^3}
M_{ijk}\left(\xi\right)v^i\left(q,s\right)v^j\left(\xi-q,s\right)\, ds,
\end{eqnarray}
for $\xi\in \mathbb{Z}^3$, and where $M_{ijk}\left(\xi\right)$ satisfies the bound 
\[
\left|M_{ijk}\left(\xi\right)\right|\leq \left|\xi\right|.
\]

To solve this problem it is usual to consider the following iteration scheme
\begin{eqnarray*}
v_{n+1}^k\left(\xi,t\right)&=&
\psi^k\left(\xi\right)\exp\left(-\left|\xi\right|^2 t\right)\\
&&
+\int_{0}^t\exp\left(-\left|\xi\right|^2\left(t-s\right)\right)\sum_{\mathbf{q}\in\mathbb{Z}^3}
M_{ijk}\left(\xi\right)v_n^i\left(q,s\right)v_n^j\left(\xi-q,s\right)\, ds.
\end{eqnarray*}
In what follows we will show the convergence of this method for small initial conditions on 
$\Phi\left(2\right)$ (for the definition of the space $\Phi\left(2\right)$ see \cite{Cortissoz}).
More exactly we will show that
\begin{theorem}
There exists an $\epsilon>0$ such that if $\left\|\psi\right\|_2<\epsilon$, then
(\ref{FourierNS}) has a global regular solution with initial condition $\psi$.
\end{theorem}

The main purpose on writing this note is for it to serve as a complement to our
paper \cite{Cortissoz}, and to show that the free divergence condition, neither the fact of
considering Leray-Hopf weak solutions is an issue for the proofs presented in that paper.  
 
\section{Existence}

We start with two auxiliary results,
\begin{lemma}
There exists an $\epsilon>0$ such that if $\left\|\psi\right\|<\epsilon$ then
the sequence $v_n^k\left(\xi,t\right)$ is uniformly bounded on $\left[0,T\right]$ for $\xi$
fixed.
\end{lemma}
\begin{proof}
To proof this fact it is enough to show that if 
\[
\left|v_{n}^k\left(\xi,t\right)\right|\leq \frac{\epsilon}{\left|\xi\right|^2}
\]
then 
\begin{equation}
\label{sumbound}
\left|\sum_{\mathbf{q}\in\mathbb{Z}^3}
M_{ijk}\left(\xi\right)v_n^i\left(q,s\right)v_n^j\left(\xi-q,s\right)\right|\leq c\epsilon^2,
\end{equation}
where $c$ is a universal constant, 
because then we would have, for any $t\geq 0$ and $\epsilon>0$ small enough,
\begin{eqnarray*}
\left|v_{n+1}^k\left(\xi,t\right)\right|&\leq&
\left|\psi^{k}\left(\xi\right)\right|\exp\left(-\left|\xi\right|^2 t\right)+
c\int_0^t \exp\left(t-s\right)\epsilon^2 \,ds\\
&\leq& \frac{\epsilon}{\left|\xi\right|^2}\exp\left(-\right|\xi\left|^2 t\right)+
\frac{c\epsilon^2}{\left|\xi\right|^2}\left(1-\exp\left(-\left|\xi\right|^2 t\right)\right)\\
&\leq& \frac{\epsilon}{\left|\xi\right|^2}\exp\left(-\right|\xi\left|^2 t\right)+
\frac{\epsilon}{\left|\xi\right|^2}\left(1-\exp\left(-\left|\xi\right|^2 t\right)\right)=
\frac{\epsilon}{\left|\xi\right|^2}.
\end{eqnarray*}

We proceed to show the validity of (\ref{sumbound}). Write
\[
\sum_{\mathbf{q}\in\mathbb{Z}^3}
M_{ijk}\left(\xi\right)v_n^i\left(q,s\right)v_n^j\left(\xi-q,s\right)= I+II+III,
\]
where
\[
I=\sum_{1\leq\left|q\right|\leq 2\left|\xi\right|,1\leq\left|\xi-q\right|\leq \frac{\left|\xi\right|}{2}}
M_{ijk}\left(\xi\right)v_n^i\left(\xi,t\right)v_n^j\left(\xi-q,t\right),
\]
\[
II=\sum_{1\leq\left|q\right|\leq 2\left|\xi\right|,\left|\xi-q\right|>\frac{\left|\xi\right|}{2}}
M_{ijk}\left(\xi\right)v_n^i\left(\xi,t\right)v_n^j\left(\xi-q,t\right),
\]
and
\[
III=\sum_{\left|q\right|>2\left|\xi\right|}
M_{ijk}\left(\xi\right)v_n^i\left(\xi,t\right)v_n^j\left(\xi-q,t\right).
\]
To estimate $I$ observe that if $\left|\xi-q\right|\leq \frac{\left|\xi\right|}{2}$, then 
$\left|q\right|\geq\frac{\left|\xi\right|}{2}$. Therefore, using that
\[
\left|M_{ijk}\left(\xi\right)\right|\leq c\left|\xi\right|
\]
and the elementary inequality
\begin{equation}
\label{basicineq0}
\sum_{1\leq\left|q\right|<r}\frac{1}{\left|q\right|^2}\leq cr
\end{equation}
(where $c$ is a universal constant) we can bound as follows,
\begin{eqnarray*}
\left|I\right|&\leq& c\left|\xi\right|\frac{\epsilon^2}{\left|\xi\right|^2}
\sum_{1\leq\left|\xi-q\right|\leq \frac{\left|\xi\right|}{2}}\frac{1}{\left|\xi-q\right|^2}\\
&\leq& \frac{c\epsilon^2}{\left|\xi\right|}\frac{\left|\xi\right|}{2}=c\epsilon^2.
\end{eqnarray*}
$II$ can be estimated in the same way, so we also obtain
\[
\left|II\right|\leq c\epsilon^2.
\]
To estimate $III$, first notice that $\left|q\right|>2\left|\xi\right|$ implies that
$\left|\xi-q\right|\geq \frac{1}{2}\left|q\right|$. Hence, using the inequality
\begin{equation}
\label{basicineq}
\sum_{\left|q\right|\geq r}\frac{1}{\left|q\right|^4}\leq \frac{c}{r},
\end{equation}
we can bound as follows,
\begin{eqnarray*}
\left|III\right|&\leq& c\left|\xi\right|\epsilon^2\sum_{\left|q\right|>2\left|\xi\right|}
\frac{1}{\left|q\right|^2}\frac{1}{\left|\xi-q\right|^2}\\
&\leq&c\left|\xi\right|\epsilon^2\sum_{\left|q\right|>2\left|\xi\right|}\frac{1}{\left|q\right|^4}\\
&\leq& c\left|\xi\right|\epsilon^2\frac{1}{\left|\xi\right|}=c\epsilon^2.
\end{eqnarray*}

This shows the lemma.
\end{proof}

\begin{lemma}
If there is an $\epsilon>0$ such that the sequence $v_n^k\left(\xi,t\right)$ satisfies
\[
\left\|v_n^k\left(t\right)\right\|_2<\epsilon \quad\mbox{for all}\quad t\in\left[0,T\right]
\]
The sequence $v_n^k\left(\xi,t\right)$ is equicontinuous on $\left[0,T\right]$ for $\xi$ fixed.
\end{lemma}

\begin{proof}

Let $t_1,t_2 \in \left(\rho,T\right)$, $t_2>t_1$.Then 
we estimate for $\xi$ fixed
\[
\left|v_{n+1}^k\left(\xi,t_2\right)-v_{n+1}^k\left(\xi,t_1\right)\right|\leq I+II+III
\]
where
\[
I=\left|\psi^k\left(\xi\right)\right|\left|\exp\left(-\left|\xi\right|^2 t_2\right)
-\exp\left(-\left|\xi\right|^2 t_1\right)\right|,
\]
\begin{eqnarray*}
II&=&\int_{0}^{t_1}\left|\exp\left(-\left|\xi\right|^2\left(t_2-s\right)\right)
-\exp\left(-\left|\xi\right|^2\left(t_1-s\right)\right)\right| \\
&&\sum_{\mathbf{q}\in\mathbb{Z}^3}
\left|M_{ijk}\left(\xi\right)v_n^i\left(q,s\right)v_n^j\left(\xi-q,s\right)\right|\, ds
\end{eqnarray*}
and
\[
III=\int_{t_1}^{t_2}\exp\left(-\left|\xi\right|^2\left(t-s\right)\right)\sum_{\mathbf{q}\in\mathbb{Z}^3}
\left|M_{ijk}\left(\xi\right)v_n^i\left(q,s\right)v_n^j\left(\xi-q,s\right)\right|\, ds.
\]
Let us bound each of the previous expressions,
\begin{eqnarray*}
I&=&\left|\exp\left(-\left|\xi\right|^2 t_1\right)\right|
\left|1-\exp\left(-\left|\xi\right|^2\left(t_2-t_1\right)\right)\right|\\
&\leq&\frac{\epsilon}{\left|\xi\right|^2}\left|\xi\right|^2\left|t_2-t_1\right|=\epsilon\left|t_2-t_1\right|,
\end{eqnarray*}
\begin{eqnarray*}
II&\leq&\int_{0}^{t_1}\left|\exp\left(-\left|\xi\right|^2\left(t_2-s\right)\right)
-\exp\left(-\left|\xi\right|^2\left(t_1-s\right)\right)\right|\epsilon^2\,ds\\
&=&\int_{\tau_n}^{t_1}\exp\left(-\left|\xi\right|^2\left(t_1-s\right)\right)
\left|1-\exp\left(-\left|\xi\right|^2\left(t_2-t_1\right)\right)\right|\epsilon^2\,ds\\
&\leq&\left|\xi\right|^2\left(t_2-t_1\right)\frac{1}{\left|\xi\right|^2}
\left|1-\exp\left(-\left|\xi\right|^2 t_1\right)\right|,
\end{eqnarray*}
\begin{eqnarray*}
III&\leq&\int_{t_1}^{t_2}\epsilon^2\exp\left(-\left|\xi\right|^2\left(t_2-s\right)\right)\,ds\\
&\leq&\frac{\epsilon^2}{\left|\xi\right|^2}\left|1-\exp\left(-\left|\xi\right|^2\left(t_2-t_1\right)\right)\right|
\leq\frac{1}{\left|\xi\right|^2}
\epsilon^2\left|\xi\right|^2\left|t_2-t_1\right|,
\end{eqnarray*}
and hence
\[
\left|v^k_{n+1}\left(\xi,t_2\right)-v^k_{n+1}\left(\xi,t_1\right)\right|<C\left(\epsilon\right)\left(t_2-t_1\right)
\]
for $n\geq 0$, and the lemma is proved.
\end{proof}

The previous Lemmas via the theorem of Arzela-Ascoli, using Cantor's diagonal procedure, show
that there is a well defined $v\in \Phi\left(2\right)$ defined on 
$\left[0,T\right]$ such that,
\begin{eqnarray}
\label{fundamentalequation}
v^k\left(\xi,t\right)&=&
v^k\left(\xi,t\right)\exp\left(-\left|\xi\right|^2 t\right)\\\notag
&&
+\int_{0}^t\exp\left(-\left|\xi\right|^2\left(t-s\right)\right)\sum_{\mathbf{q}\in\mathbb{Z}^3}
M_{ijk}\left(\xi\right)v^i\left(q,s\right)v^j\left(\xi-q,s\right)\, ds
\end{eqnarray}
Let us give a proof of this. To simplify notation, let us assume that the sequence converging uniformly
on $\left[0,T\right]$ for each $\xi$ is the sequence $v_n\left(\xi,t\right)$. By what we have shown, there exists
a $D$ not depending on $t$, $\xi$ or $n$ such that 
\[
\left|v_n^j\left(\xi,t\right)\right|\leq \frac{D}{\left|\xi\right|^2}.
\]
Let $\xi$ be fixed, and let $\eta>0$ arbitrary. the previous estimate allows us to choose a $Q$ such that 
\[
\left|\sum_{\left|q\right|\geq Q}M_{ijk}\left(\xi\right)v_n^i\left(\xi,t\right)v_n^j\left(\xi,t\right)\right|
\leq \eta.
\]
and also that the same inequality is valid with $v_n$ replaced by $v$ (this can be done since the
choice of $Q$ only depends on $D$).
Hence we have
\begin{eqnarray*}
\left|v_{n+1}^k\left(\xi,t\right)\right.&-&\psi^k\left(\xi\right)\exp\left(-\left|\xi\right|^2 t\right)\\
&-&\left.\int_{0}^t
\exp\left(-\left|\xi\right|^2\left(t-s\right)\right)\sum_{1\leq\left|q\right|<Q}
M_{ijk}\left(\xi\right)v_n^i\left(q,s\right)v_n^j\left(\xi-q,s\right)\, ds\right|\\
&\leq& \eta
\end{eqnarray*}
Taking $n\rightarrow \infty$, we obtain
\begin{eqnarray*}
\left|v^k\left(\xi,t\right)\right.&-&\psi^k\left(\xi\right)\exp\left(-\left|\xi\right|^2 t\right)\\
&-&\left.\int_{0}^t
\exp\left(-\left|\xi\right|^2\left(t-s\right)\right)\sum_{1\leq\left|q\right|<Q}
M_{ijk}\left(\xi\right)v^i\left(q,s\right)v^j\left(\xi-q,s\right)\, ds\right|\\
&\leq& \eta
\end{eqnarray*}
and from this follows that
\begin{eqnarray*}
\left|v^k\left(\xi,t\right)\right.&-&\psi^k\left(\xi\right)\exp\left(-\left|\xi\right|^2 t\right)\\
&-&\left.\int_{0}^t
\exp\left(-\left|\xi\right|^2\left(t-s\right)\right)\sum_{q\in\mathbb{Z}^3}
M_{ijk}\left(\xi\right)v^i\left(q,s\right)v^j\left(\xi-q,s\right)\, ds\right|\\
&\leq& 2\eta.
\end{eqnarray*}
Since $\eta>0$ is arbitrary, our claim is proved. 

\section{Regularity}

We shall show now that the solutions produced by the iteration scheme are regular under
certain smallness condition. Indeed, we have
\begin{theorem}
\label{regularity}
Let $v\in L^{\infty}\left(0,T;\Phi\left(2\right)\right)$ be a solution to (\ref{FourierNS}). There
exists an $\epsilon>0$ such that if there is a $k_{-1}$ for which $v$ satisfies
\begin{equation}
\sup_{\left|\xi\right|\geq k_{-1}}\left|\xi\right|^2\left|v^k\left(\xi,t\right)\right|<\epsilon
\quad\mbox{for all} \quad t\in\left(0,T\right)
\end{equation}
then $v$ is smooth.
\end{theorem}

To prove Theorem \ref{regularity} we will need to estimate term
\[
\sum M_{ijk}\left(\xi\right)u^i\left(q\right)u^j\left(\xi-q\right).
\]
This is the content of Lemma \ref{regularity1}. 
But before we state and prove Lemma \ref{regularity1} and in order to express
our estimates in a convenient way we will define to sequences of numbers.
Namely
\[
\left\{
\begin{array}{l}
\mu_0=1\quad \mu_1=1\\
\mu_{n+1}=2\mu_n-1,\quad n\geq 2 
\end{array}
\right.
\]
and 
\[
k_n=\frac{1}{\epsilon^{2^n}}k_0
\]
where $k_0$ is such that 
\[
\frac{k_{-1}}{k_0}\cdot D < \min \left\{\epsilon, \frac{1}{2}\right\}
\]
and $D= \sup_{\left(0,T\right)}\left\|u\left(t\right)\right\|$.

We are now ready to estate and prove,
\begin{lemma}
\label{regularity1}
Assume  that for all $\xi$ such that $\left|\xi\right|\geq k_{-1}$ 
\[
\left|v^k\left(\xi,s\right)\right|\leq \frac{\epsilon}{\left|\xi\right|^2}
\]
and if $\left|\xi\right|\geq k_m$
\[
\left|v^k\left(\xi,s\right)\right|\leq \frac{\epsilon^{\mu_m}}{\left|\xi\right|^2}
\]
Then for $\left|\xi\right|\geq k_{m+1}$ it holds that,
\[
\left|\sum_{q\in\mathbb{Z}^3}M_{ijk}\left(\xi\right)v^i\left(q,s\right)v^j\left(\xi-q,s\right)\right|
\leq \epsilon^{\mu_{m+1}}.
\]
\end{lemma}

\begin{proof}
First recall that $\left|M_{ijk}\left(\xi\right)\right|\leq c\left|\xi\right|$.
\begin{eqnarray}
\label{firstineq}
I&\leq& 
\left|\xi\right|\sum_{1\leq\left|q\right|<k_{-1}}\left|v^i\left(q,s\right)v^j\left(\xi-q,s\right)\right|+
\left|\xi\right|\sum_{k_{-1}\leq\left|q\right|< k_m}\left|v^i\left(q,s\right)v^j\left(\xi-q,s\right)\right|\\ \notag
&&+
\left|\xi\right|\sum_{\left|q\right|\geq  k_m}\left|v^i\left(q,s\right)v^j\left(\xi-q,s\right)\right|
\end{eqnarray}
We estimate the first sum. Observe that $k_{-1}\leq \frac{\left|\xi\right|}{2}$, so if $\left|q\right|<k_{-1}$,
we must have $\left|\xi-q\right|\geq \frac{\left|\xi\right|}{2}$. Hence, using the elementary inequality
(\ref{basicineq0}), we can bound
\begin{eqnarray*}
\sum_{1\leq\left|q\right|<k_{-1}}\left|v^i\left(q,s\right)v^j\left(\xi-q,s\right)\right|
&\leq&
\frac{4\epsilon^{\mu_m}}{\left|\xi\right|^2}\sum_{1\leq\left|q\right|<k_{-1}}\frac{D}{\left|q\right|^2}\\
&\leq& 4c\epsilon^{\mu_m}\frac{k_{-1}}{k_m}\leq 4c\epsilon^{2\mu_m}
\end{eqnarray*}

To estimate the second sum, notice 
that if $\left|\xi\right|\geq k_{m+1}$ and $\left|q\right|\leq k_m$, then 
$\left|\xi-q\right|\geq\frac{\left|\xi\right|}{2}\geq k_m$. All this said,
using inequality (\ref{basicineq0}) again we obtain, 
\begin{eqnarray*}
\sum_{1\leq\left|q\right|< k_m}\left|v^i\left(q,s\right)v^j\left(\xi-q,s\right)\right|
&\leq& 
\frac{4\epsilon^{\mu_m}}{\left|\xi\right|^2}\sum_{1\leq\left|q\right|<k_m}\frac{\epsilon}{\left|q\right|^2}\\
&\leq& \frac{4\epsilon^{\mu_m}}{\left|\xi\right|^2}\epsilon k_m
\end{eqnarray*}
Observe now that $\frac{k_m}{k_{m+1}}\leq \epsilon^{2^m}\leq \epsilon^{\mu_m}$. This yields
the bound,
\[
\sum_{1\leq\left|q\right|< k_m}\left|v^i\left(q,s\right)v^j\left(\xi-q,s\right)\right|
\leq  \frac{4\epsilon^{\mu_m}}{\left|\xi\right|}\frac{k_m}{k_{m+1}}
\leq \frac{4\epsilon^{2\mu_m}}{\left|\xi\right|}
\]
To estimate the second sum on the righthanside of (\ref{firstineq}) we split it into three sums, namely
\begin{eqnarray}
\label{firsteq}
\sum_{\left|q\right|\geq k_m}\left|v^i\left(q,s\right)v^j\left(\xi-q,s\right)\right|
&=&\sum_{k_m\leq \left|q\right|<\frac{\left|\xi\right|}{2}}
\left|v^i\left(q,s\right)v^j\left(\xi-q,s\right)\right|\\ \notag
&&
+\sum_{\frac{\left|\xi\right|}{2}\leq \left|q\right|<2\left|\xi\right|}
\left|v^i\left(q,s\right)v^j\left(\xi-q,s\right)\right|\\ \notag
&&
+
\sum_{\left|q\right|\geq 2\left|\xi\right|}
\left|v^i\left(q,s\right)v^j\left(\xi-q,s\right)\right|
\end{eqnarray}
Estimating the three sums on the right hand side separately. Observe that
if $\left|q\right|\leq \frac{\left|\xi\right|}{2}$ then we must have 
$\left|\xi-q\right|\geq \frac{\left|\xi\right|}{2}>k_m$. Therefore, using inequality
(\ref{basicineq0}), we get

\begin{eqnarray*}
\sum_{k_m\leq\left|q\right|<\frac{\left|\xi\right|}{2}}
\left|v^i\left(q,s\right)v^j\left(\xi-q,s\right)\right|&\leq&
\frac{4\epsilon^{2\mu_m}}{\left|\xi\right|^2}\sum_{1\leq\left|q\right|<\frac{\left|\xi\right|}{2}}
\frac{1}{\left|q\right|^2}\\
&\leq& \frac{4\epsilon^{2\mu_m}}{\left|\xi\right|}
\end{eqnarray*}

To estimate the second sum we split it into two sums,
\begin{eqnarray*}
\sum_{\frac{\left|\xi\right|}{2}\leq \left|q\right|<2\left|\xi\right|}
\left|v^i\left(q,s\right)v^j\left(\xi-q,s\right)\right|&=&
\sum_{\frac{\left|\xi\right|}{2}\leq \left|q\right|<2\left|\xi\right|,k_m\leq\left|\xi-q\right|}
\left|v^i\left(q,s\right)v^j\left(\xi-q,s\right)\right|\\
&&+
\sum_{\frac{\left|\xi\right|}{2}\leq \left|q\right|<2\left|\xi\right|,\left|\xi-q\right|<k_m}
\left|v^i\left(q,s\right)v^j\left(\xi-q,s\right)\right|
\end{eqnarray*}

Estimating the first sum on the righthandside of the previous equality,
\begin{eqnarray*}
\sum_{\frac{\left|\xi\right|}{2}\leq \left|q\right|<2\left|\xi\right|,\left|\xi-q\right|\geq k_m}
\left|v^i\left(q,s\right)v^j\left(\xi-q,s\right)\right|
&\leq&
\frac{4\epsilon^{2\mu_m}}{\left|\xi\right|^2}\sum_{1\leq\left|\xi-q\right|<3\left|\xi\right|}
\frac{1}{\left|\xi-q\right|^2}\\
&\leq& \frac{12\epsilon^{2\mu_m}}{\left|\xi\right|}
\end{eqnarray*}
The estimation of the second sum proceeds in exactly
the same way as the estimation of the first sum
on the right hand side of (\ref{firstineq}), and hence we obtain
\begin{equation*}
\sum_{\frac{\left|\xi\right|}{2}\leq \left|q\right|<2\left|\xi\right|,\left|\xi-q\right|<k_m}
\left|v^i\left(q,s\right)v^j\left(\xi-q,s\right)\right|
\leq \frac{4\epsilon^{2\mu_m}}{\left|\xi\right|}.
\end{equation*}

Now we estimate the third sum in the righthandside of (\ref{firsteq}). Using that
$\left|q\right|\geq 2\left|\xi\right|$ implies that $\left|\xi-q\right|\geq \frac{1}{2}\left|q\right|$,
and inequality (\ref{basicineq})
we can bound,
\begin{eqnarray*}
\sum_{\left|q\right|\geq 2\left|\xi\right|}
\left|v^i\left(q,s\right)v^j\left(\xi-q,s\right)\right|&\leq&
4\epsilon^{2\mu_m}\sum_{\left|q\right|\geq 2\left|\xi\right|}\frac{1}{\left|q\right|^4}
\leq \frac{4\epsilon^{2\mu_m}}{\left|\xi\right|}
\end{eqnarray*}

Putting all the previous estimations together, we arrive at
\begin{eqnarray*}
\left|\sum_{q\in\mathbb{Z}^3}M_{ijk}\left(\xi\right)v^i\left(q,s\right)v^j\left(\xi-q,s\right)\right|
&\leq& \left|\xi\right|\left(\frac{28\epsilon^{\mu_m}}{\left|\xi\right|}\right),
\end{eqnarray*}
and if we assume $0<\epsilon<\frac{1}{28}$, the previous inequality reads as
\begin{equation*}
\left|\sum_{q\in\mathbb{Z}^3}M_{ijk}\left(\xi\right)v^i\left(q,s\right)v^j\left(\xi-q,s\right)\right|
\leq \epsilon^{2\mu_m-1}\leq \epsilon^{\mu_{m+1}},
\end{equation*}
and the Lemma is proved.
\end{proof}

\subsection{Proof of Theorem \ref{regularity}} 
Given $0<\rho<T$, we will first show that for a constant $K\left(\rho\right)$, there exists
a constant $D$ such that if $\left|\xi\right|\geq K\left(\rho\right)$ then
\[
\left|v^k\left(\xi, t\right)\right|\leq\frac{D}{\left|\xi\right|^{2+\frac{1}{4}}}\quad
\mbox{if}\quad t>\rho.
\]

Define
\[
\tau_m=\rho-\frac{\rho}{2^m}.
\]
We will show by induction that
\begin{equation}
\tag{P}
\label{property}
v^k\left(\xi,t\right)\leq \frac{\epsilon^{\mu_n}}{\left|\xi\right|^2}
\quad\mbox{if} \quad t>\tau_n\quad\mbox{and}\quad \left|\xi\right|\geq k_n.
\end{equation}
For $n=0$, our choice of $k_0$ guarantees that (\ref{property}) holds. Assume
that (\ref{property}) holds for $n=m$. First observe that $v$ satisfies 
\begin{eqnarray*}
v^k\left(\xi,t\right)&=&
v^k\left(\xi,\tau_n\right)\exp\left(-\left|\xi\right|^2\left(t-\tau_n\right)\right)\\\notag
&&
+\int_{\tau_n}^t\exp\left(-\left|\xi\right|^2\left(t-s\right)\right)\sum_{\mathbf{q}\in\mathbb{Z}^3}
M_{ijk}\left(\xi\right)v^i\left(q,s\right)v^j\left(\xi-q,s\right)\, ds.
\end{eqnarray*}

Using this identity, we bound as follows,
\begin{eqnarray*}
v^k\left(\xi,t\right)&\leq& v^k\left(\xi,\tau_m\right)\exp\left(-\left|\xi\right|^2\left(t-\tau_m\right)\right)+
\int_{\tau_m}^t \exp\left(-\left|\xi\right|^2\left(t-s\right)\right)\epsilon^{2\mu_m}\,ds\\
&\leq& \frac{\epsilon^{\mu_m}}{\left|\xi\right|^2}\exp\left(-k_{m+1}\left(\tau_{m+1}-\tau_m\right)\right)\\
&&+
\frac{\epsilon^{2\mu_m}}{\left|\xi\right|^2}
\left(\exp\left(-\left|\xi\right|^2 \tau_m\right)-\exp\left(-\left|\xi\right|^2 t\right)\right)\\
&\leq& \frac{\epsilon^{\mu_m}}{\left|\xi\right|^2}+\frac{\epsilon^{2\mu_m}}{\left|\xi\right|^2}.
\end{eqnarray*}
From this last bound it follows that if $t\geq \rho>\rho-\frac{\rho}{2^m}$, 
then if $k_m\leq\left|\xi\right|<k_{m+1}$ it holds that
\[
\left|v^k\left(\xi,t\right)\right|\leq \frac{\epsilon^{\mu_m}}{\left|\xi\right|^2}.
\]
Since $\mu_m\geq 2^{n-1}$ and $k_m=\frac{k_0}{\epsilon^{2^m}}$, 
it is easy to check that $\epsilon^{\mu_m}\leq \frac{k_0^{\frac{1}{4}}}{\left|\xi\right|^{\frac{1}{4}}}$.
Hence for all $t\geq \rho$ the following estimate holds,
\[
\left|v^k\left(\xi,t\right)\right|\leq \frac{D}{\left|\xi\right|^{2+\frac{1}{4}}}.
\]

The following Lemma will then finish the proof of Theorem \ref{regularity}.
\begin{lemma}
Let $v$ be a solution to (FNS) such that for all $t\in \left(0,T\right)$ satisfies
\[
\left|v^k\left(\xi,t\right)\right|\leq \frac{D}{\left|\xi\right|^{2+\eta}}
\]
with $D$ and $\eta>0$ independent of $t$. Then $v$ is smooth.
\end{lemma}
\begin{proof}
Let $\rho>0$. Under the hypothesis of the Lemma,
we will show that there exists a constant $K:=K\left(\rho\right)$ such that
if $t>T$ and $\left|\xi\right|>K$, then for a constant $E$ independent of time,
\[
\left|v^k\left(\xi,t\right)\right|\leq \frac{E}{\left|\xi\right|^{2+\min\left(\frac{1}{2},\frac{3}{2}\eta\right)}}.
\]
Since $\rho>0$ is arbitrary, a finite number of applications of the previous claim shows
that for any $\rho>0$, the Fourier transform of $v$ decays faster than any polynomial, and
this shows the lemma. 

First, we will estimate the term
\[
S=\sum_{q\in \mathbb{Z}^3} M_{ijk}\left(\xi\right)v^i\left(\xi,s\right)v^j\left(\xi,s\right)
\]
under the hypotesis of the lemma. In order to do this we write,
\[
S=I_a+I_b+II_a+II_b+III_a+III_b+IV_a+IV_b
\]
where
\begin{equation*}
I_a=\sum_{1\leq\left|q\right|\leq\sqrt{\left|\xi\right|}}M_{ijk}\left(\xi\right)v^i\left(\xi,s\right)v^j\left(\xi,s\right),
\end{equation*}
\begin{equation*}
II_a=\sum_{\sqrt{\left|\xi\right|}<q\leq \frac{\left|\xi\right|}{2}}
M_{ijk}\left(\xi\right)v^i\left(\xi,s\right)v^j\left(\xi,s\right),
\end{equation*}
\begin{equation*}
III_a=\sum_{\left|q\right|\geq\frac{\left|\xi\right|}{2}, 1\leq\left|\xi-q\right|< 2\left|\xi\right|}
M_{ijk}\left(\xi\right)v^i\left(\xi,s\right)v^j\left(\xi,s\right),
\end{equation*}
and
\begin{equation*}
IV_a=\sum_{\left|q\right|\geq\frac{\left|\xi\right|}{2}, \left|\xi-q\right|\geq 2\left|\xi\right|}
M_{ijk}\left(\xi\right)v^i\left(\xi,s\right)v^j\left(\xi,s\right)
\end{equation*}
The corresponding $I_b,II_b,III_b$ and $IV_b$ are the same as
their $a$ counterparts, except that the role of $q$ and $\xi-q$ is interchanged.
Noticed that by the triangular inequality not both $q$ and $\xi-q$ can be less than $\frac{\left|\xi\right|}{2}$,
and hence all possible cases are covered.

Since $\left|q\right|<\sqrt{\left|\xi\right|}<\frac{\left|\xi\right|}{2}$, and hence
$\left|\xi-q\right|\geq \frac{\left|\xi\right|}{2}$. Hence we have,
\begin{eqnarray*}
\left|I_a\right|&\leq& \left|\xi\right|\sum_{1\leq\left|q\right|\leq\sqrt{\left|\xi\right|}}
\frac{D}{\left|q\right|^{2+\eta}}\frac{D}{\left|\xi-q\right|^{2+\eta}}\\
&\leq&\left|\xi\right|\frac{2^{2+\eta}D^2}{\left|\xi\right|^{2+\eta}}
\sum_{1\leq\left|q\right|\leq\sqrt{\left|\xi\right|}}
\frac{D}{\left|q\right|^2}\\
&\mbox{and by inequality (\ref{basicineq0})}&\\
&\leq& \left|\xi\right|\frac{2^{2+\eta}D^2}{\left|\xi\right|^{2+\eta}}\sqrt{\left|\xi\right|}=
\frac{2^{2+\eta}D^2}{\left|\xi\right|^{\frac{1}{2}+\eta}}.
\end{eqnarray*}
Estimating $II_a$ and $III_a$ is pretty straightforward, via the inequality
\begin{equation*}
\sum_{1\leq\left|q\right|<r} 1 \leq cr^3.
\end{equation*}
Indeed,
\begin{eqnarray*}
\left|II_a\right|&\leq& \left|\xi\right|\frac{2^{2+\eta}D}{\left|\xi\right|^{2+\eta}}\cdot
\frac{D}{\left(\sqrt{\left|\xi\right|}\right)^{2+\eta}}\left(\sum_{\left|q\right|\leq\frac{\left|\xi\right|}{2}}1\right)\\
&\leq&\frac{2^{2+\eta}D}{\left|\xi\right|^{1+\eta}}\cdot
\frac{D}{\left|\xi\right|^{1+\frac{\eta}{2}}}\left|\xi\right|^3=\frac{2^{2+\eta}D^2}{\left|\xi\right|^{\frac{3}{2}\eta}}.
\end{eqnarray*}
\begin{eqnarray*}
\left|III_a\right|&\leq&\left|\xi\right|
\sum_{\left|q\right|\geq \frac{\left|\xi\right|}{2},\frac{\left|\xi\right|}{2}\leq\left|\xi-q\right|<2\left|\xi\right|}
\frac{D}{\left|q\right|^{2+\eta}}\frac{D}{\left|\xi-q\right|^{2+\eta}}\\
&\leq&\left|\xi\right|\cdot\frac{2^{2+\eta}D^2}{\left|\xi\right|^{4+2\eta}}
\left(\sum_{1\leq\left|\xi-q\right|<2\left|\xi\right|}1\right)\\
&\leq&\frac{2^{2+\eta}}{\left|\xi\right|^{2\eta}}
\end{eqnarray*}
Finally, using that $\left|\xi-q\right|\geq 2\left|\xi\right|$ and $\left|q\right|\geq \frac{\left|\xi\right|}{2}$
imply that $\left|q\right|\geq \frac{2}{3}\left|\xi-q\right|$ and inequality (\ref{basicineq}) we can bound
$IV_a$ as follows,
\begin{eqnarray*}
\left|IV_a\right|&\leq&\left|\xi\right|
\sum_{\left|q\right|\geq\frac{\left|\xi\right|}{2}, \left|\xi-q\right|\geq2\left|\xi\right|}
\frac{D}{\left|q\right|^{2+\eta}}\frac{D}{\left|\xi-q\right|^{2+\eta}}\\
&\leq&
\frac{2^{2\eta}}{\left|\xi\right|^{2\eta}}\left(\frac{3}{2}\right)^{2+\eta}
\sum_{\left|q\right|\geq \frac{\left|\xi\right|}{2}}\frac{D^2}{\left|q\right|^4}\\
&\leq& \left|\xi\right|\frac{1}{\left|\xi\right|^{2\eta}}\frac{D^2}{\left|\xi\right|}=\frac{D^2}{\left|\xi\right|^{2\eta}}.
\end{eqnarray*}
The proof is now complete.
\end{proof}

\end{document}